\documentclass[12pt,a4paper]{amsart}

\usepackage{amssymb,amscd,latexsym,amsxtra,dsfont,enumerate}

\usepackage{amsfonts,bbold}
\usepackage[mathscr]{eucal}
\usepackage{calrsfs}
\usepackage{fge}
\usepackage{pst-all}
\usepackage{mathtools}
\usepackage{scalerel}

\definecolor{britishracinggreen}{rgb}{0.0, 0.26, 0.15}
\definecolor{amaranth}{rgb}{0.9, 0.17, 0.31}

\newtheorem{thm}{Theorem}
\newtheorem{cor}[thm]{Corollary}
\newtheorem{lem}[thm]{Lemma}
\newtheorem{prop}[thm]{Proposition}

\newtheorem{remark}{Remark}

\theoremstyle{definition}
\newtheorem{defn}{Definition}
\newtheorem{exmp}[thm]{Example}
\newtheorem{prob}[defn]{Problem}

\newcommand{\set}[2]{\ensuremath{\left\{ #1\,:\; #2\right\}}}

\newcommand{\sss}[1]{{\scriptscriptstyle #1}}

\newcommand{\norm}[2]{\Vert #2\Vert_{\sss{#1}}}

\newcommand{\saM}{\ensuremath{M_{sa}}}

\newcommand{\pp}{\mathscr P}
\newcommand{\qq}{\mathscr Q}
\newcommand{\aaa}{\mathscr A}

\newcommand{\ttt}{\mathscr T}
\newcommand{\bb}{\mathscr B}
\newcommand{\dd}{\mathscr D}

\newcommand{\oo}{\mathscr O}
\newcommand{\mm}{\mathscr M}
\newcommand{\nn}{\mathscr N}

\newcommand{\ff}{\mathscr F}

\newcommand{\xx}{\mathscr X}

\newcommand{\N}{\mathds N}
\newcommand{\R}{\mathds R}
\newcommand{\C}{\mathds C}

\newcommand{\Z}{\mathds Z}

\mathchardef\mhyphen="2D

\newcommand{\msp}{(X,\aaa,\mu)}
\newcommand{\integral}[1]{\int #1\,{\rm d}\mu}

\newcommand{\rintegral}[2]{\int_{#1} #2 \,{\rm d}\mu}

\newcommand{\muaew}{$\mu$-{\rm a.e.}}

\newcommand{\ca}[1]{{\rm ca}(#1)}
\newcommand{\converges}[1]{\mathop{\overset{\text{#1}}{\longrightarrow}}}

\newcommand\restr[2]{{
  \left.\kern-\nulldelimiterspace 
  #1 
  \vphantom{\big|} 
  \right|_{#2} 
  }}

\newcommand{\mysetminus}{\mathbin{\fgebackslash}}

\DeclareMathOperator{\spn}{span}

\begin{document}

\title{On different modes of order convergence and some applications}

\author{Kevin Abela}
\address{
Kevin Abela \\
Department of Mathematics \\
Faculty of Science \\
University of Malta \\
Msida MSD 2080, Malta} \email {kevin.abela.11@um.edu.mt}

\author{Emmanuel Chetcuti}
\address{
Emmanuel Chetcuti,
Department of Mathematics\\
Faculty of Science\\
University of Malta\\
Msida MSD 2080  Malta} \email {emanuel.chetcuti@um.edu.mt}

\author{Hans Weber}
\address{Hans Weber\\
Dipartimento di matematica e informatica\\
Universit\`a degli Studi di Udine\\
1-33100 Udine\\
Italia}
\email{hans.weber@uniud.it}

\date{\today}
\begin{abstract}
Different notions for order convergence have been considered by various authors.  Associated to every notion of order convergence corresponds a topology, defined by taking as the closed sets those subsets of the poset satisfying that no net in them order converges to a point that is outside of the set.  We shall give  a thorough overview of these different notions and provide a systematic comparison of the associated topologies.  Then, in the last section we shall give an application of this study by giving a result on von Neumann algebras complementing the study started in  \cite{ChHaWe}.   We show that for every atomic von Neumann algebra (not necessarily $\sigma$-finite) the restriction of the order topology to bounded parts of $M$ coincides with the restriction of the $\sigma$-strong topology $s(M,M_\ast)$.  We recall that the methods of \cite{ChHaWe} rest heavily on the assumption of $\sigma$-finiteness.  Further to this,  for a semi-finite measure space, we shall give a complete picture of the relations between the topologies on $L^\infty$ associated with the duality $\langle L^1, L^\infty\rangle$ and its order topology.
\end{abstract}
\maketitle

\section{Introduction}

Order convergence in partially ordered sets has for long been studied.  One finds different definitions for order convergence that all reduce -- for complete lattices -- to the requirement $\limsup=\liminf$.  Associated to each notion of order convergence there exists a topology, defined by taking as the closed sets those subsets of the poset satisfying that no net in them order converges to a point that is outside of the set.  This topology is called the order topology and being the finest topology preserving order convergence,  it is intrinsic to the poset.  The simplest example is the set of real numbers. The order topology on the real line is the standard topology and if a sequence converges w.r.t. the order topology, it is also order convergent.  Caution is required here because in general it is not so that the nets that converge w.r.t. the order topology are also order convergent.  Order convergence is not topological in general.  The question of characterizing those posets for which  order convergence is topological has received considerable attention. However, we shall not follow this path.  Our work continues on the work in \cite{BuChWe2012,ChHaWe,ChHa2020,BuChWe2020,Bohata2018} in which the order topology associated to a partially ordered structure arising naturally in the realm of Functional Analysis is studied and related to the functional-analytic properties of the structure.

 A systematic treatment of the order topology on various  structures associated with a von Neumann algebra  was carried out in \cite{ChHaWe}. In that case, the order topology comes from the standard  operator order defined on the self-adjoint operators acting on a Hilbert space $H$; namely, for self-adjoint operators $T$ and $S$, we have  $T\le S$  if the operator $S-T$ has a positive spectrum.  It was shown that the order topology on  the self-adjoint part of a $\sigma$-finite von Neumann algebra,  albeit far from being a  linear topology in general, on bounded parts, coincides  with the strong operator topology. The proof is based on the Noncommutative Egoroff Theorem and rests heavily on the assumption of $\sigma$-finiteness.  Besides,  finite $\sigma$-finite von Neumann algebras were characterized in terms of the order topological properties of the projection lattice.

 The aim of the present paper is two-fold.  In the first section we shall present a thorough overview of the different definitions for order convergence that are found in the literature and also provide a systematic comparison of the associated order topologies. In the second section we shall consider some applications complementing the work done in \cite{ChHaWe, ChHa2020}.  We shall first consider a semi-finite measure and study the relation between the locally convex topologies associated with $L^\infty$ and its order topology.  Then we shall make use of the results obtained in the first section to show that for atomic von Neumann algebras the underlying assumption of $\sigma$-finiteness is redundant in \cite[Corollary 4.4]{ChHaWe}.

\section{On relations between different notions of order convergence in a partially ordered set}

Let $(\pp,\le)$ be a partially ordered set.  A subset $\dd$ of $\pp$ is \emph{directed} provided it is nonempty and every finite subset of $\dd$ has an upper bound in $\dd$.  Dually, a nonempty subset $\ff$ of $\pp$ is \emph{filtered} if every finite subset of $\ff$ has a lower bound in $\ff$.  If the set of upper bounds of a subset $\xx$ of $\pp$ has a least element, we call this element the \emph{supremum} (or least upper bound) of $\xx$ and denote it by $\sup X$ or $\vee X$.  Similarly, the \emph{infimum}  (or greatest lower bound) is written as $\inf X$ or $\wedge X$.  The poset $(\pp,\le)$ is said to be:
\begin{itemize}
  \item a \emph{lattice} if every finite subset of $\pp$ has a supremum and infimum.
  \item \emph{conditionally complete} if every subset of $\pp$ that is bounded from above has a supremum and every subset of $\pp$ that is bounded from below has an infimum.
  \item \emph{conditionally monotone complete} if every directed subset of $\pp$ that is bounded from above has a supremum and every filtered subset of $\pp$ that is bounded from below has an infimum.
\end{itemize}
A function $\varphi$ from a poset $(\pp,\le)$ into another poset $(\qq,\le)$ is said to be \emph{isotone} if $\varphi(a)\le\varphi(b)$ for every $a,b\in\pp$ satisfying $a\le b$.  \emph{Antitone} functions between posets are defined dually. When $\varphi$ is isotone and injective, and $\varphi^{-1}$ (defined on the range of $\varphi$) is again isotone, we say that $\varphi$ is an \emph{order isomorphism}. The posets $(\pp,\le)$ and $(\qq,\le)$ are said to be order isomorphic when there is a surjective order isomorphism between them.

A \emph{net} in a set $X$ is a function $\varphi$ from a directed set $\Gamma$ into $X$.  We shall write $x_\gamma$ instead of $\varphi(\gamma)$ and denote the net by $(x_\gamma)_{\gamma\in\Gamma}$.   For a net $(x_{\gamma})_{\gamma \in \Gamma}$, the subset $E_{\Gamma}(\gamma_{0}) = \{x_{\gamma} : \gamma \geq \gamma_{0} \}$ is  called a \emph{residual} of $(x_{\gamma})_{\gamma \in \Gamma}$.  If $P(x)$ is a property of the elements $x\in X$ and $P(x_\gamma)$ is true for some residual $E_\Gamma(\gamma_0)$, then we say that  the net $(x_\gamma)_{\gamma\in\Gamma}$ \emph{satisfies $P(x)$ eventually}.    The subset $\Gamma^{'} \subseteq \Gamma$ is \emph{cofinal} if for every $\gamma \in \Gamma$ there exists $\gamma_{0} \in \Gamma^{'}$ such that $\gamma \leq \gamma_{0}$.   It is easily seen that if $\Gamma'$ is cofinal in $\Gamma$, then $\Gamma'$ is directed and therefore $(x_\gamma)_{\gamma\in\Gamma'}$ is a net.  This net is said to be a \emph{cofinal subnet} of $(x_{\gamma})_{\gamma \in \Gamma}$.

A net $(x_\gamma)_{\gamma\in\Gamma}$ in a partially ordered set $\pp$ is said to be \emph{monotonic increasing} or \emph{monotonic decreasing} depending on whether the function $\gamma\mapsto x_\gamma$ is isotone or antitone, respectively.  A net is called \emph{monotone} if it is either monotonic increasing or monotonic decreasing.  If $(x_\gamma)_{\gamma\in\Gamma}$ is monotonic increasing and its supremum exists and equals $x$ we write $x_\gamma\uparrow x$.  Dually, $x_\gamma\downarrow x$ means that the net $(x_\gamma)_{\gamma\in\Gamma}$ is monotonic decreasing with infimum equal to $x$.

 Order convergence in posets has been studied since at least the time of Birkhoff.  The motivation can be traced to the squeezing lemma, well known to every student of calculus: Let $(y_n)_{n\in\N}$ be   an increasing sequence  of real numbers  and $(z_n)_{n\in\N}$ be a decreasing  sequence  of real numbers such that $y_n\le z_n$ for all $n$. Suppose that  the supremum of $(y_n)_{n\in\N}$ is the same as the infimum of $(z_n)_{n\in\N}$ and equals to $\alpha$. Then, for any other sequence, $(x_n)_{n\in\N}$, squeezed between $(y_n)_{n\in\N}$ and $(z_n)_{n\in\N}$ (i.e. $y_n\le x_n \le z_n$ for all $n$)  we have that $(x_n)_{n\in\N}$ converges to $\alpha$.  The following three different generalizations for posets can be found in the literature.  We call them O$_1$-, O$_2$- and O$_3$- convergence.  For every $a,b\in \pp$ we define $(\leftarrow,a]:=\{x\in\pp:x\le a\}$;  $[a,\rightarrow)$ is defined dually, and $[a,b]:=\{x\in\pp:a\le x\le b\}$.

 \begin{defn}
   Let $(x_\gamma)_{\gamma\in\Gamma}$ be a net and  $x$ a point in a poset $\pp$.
   \begin{enumerate}[{\rm(i)}]
     \item $(x_\gamma)_{\gamma\in\Gamma}$ is said to \emph{O$_1$-converge} to $x$ in $\pp$ if there exist two nets $(y_\gamma)_{\gamma\in\Gamma}$, $(z_\gamma)_{\gamma\in\Gamma}$ in $\pp$ such that eventually $y_\gamma\le x_\gamma\le z_\gamma$, $y_\gamma\uparrow x$ and $z_\gamma\downarrow x$.
     \item $(x_\gamma)_{\gamma\in \Gamma}$ is said to \emph{O$_2$-converge} to $x$ in $\pp$ if  there exists a directed subset $M\subset \pp$, and a filtered subset $N\subset\pp$,  such that $\vee M=\wedge N=x$, and for every $(m,n)\in M\times N$ the net is eventually contained in $[m,n]$.
     \item $(x_\gamma)_{\gamma\in \Gamma}$ is said to \emph{O$_3$-converge} to $x$ in $\pp$ if  there exist two subsets  $M$ and  $N$ of $\pp$ such that $\vee M=\wedge N=x$, and for every $ (m,n)\in M\times N$ the net is eventually contained in $[m,n]$.
   \end{enumerate}
 \end{defn}

 The notion of O$_1$-convergence can be traced back to Birkhoff \cite{Birkhoff1940} and Kantorovich \cite{Kantorovich1935}.  Note that, in contrast to O$_1$-convergence, in O$_2$-convergence, the controlling nets are not indexed by the same directed set of the original net.  O$_2$-convergence is due McShane \cite{McShane1953}.  The notion of O$_3$-convergence can be found in \cite{Wolk1961} whereby he gives references to \cite{Rennie1951} and \cite{Ward1955}.  It is easy to see that:
  \begin{itemize}
  \item For a net $(x_\gamma)_{\gamma\in\Gamma}$ in a poset $\pp$,
 \[x_\gamma\converges{\text{O$_1$}} x\ \Rightarrow\  x_\gamma\converges{\text{O$_2$}}x\ \Rightarrow\  x_\gamma\converges{\text{O$_3$}} x\qquad (x\in\pp).\]
 \item Every eventually constant net is O$_1$-convergent to its eventual value.
 \item Every cofinal subnet of an O$_i$-convergent net is O$_i$-convergent to the same limit (where $i\in\{1,2,3\}$).
 \end{itemize}
Let us verify that when $(x_\gamma)_{\gamma\in\Gamma}$ is a net in a conditionally complete lattice, the expression `$(x_\gamma)_{\gamma\in\Gamma}$ order converges to $x$'  conveys the intuitive meaning
\begin{equation}\label{e3}
  \liminf_\gamma x_\gamma=\sup_{\gamma\in\Gamma}\inf_{\gamma'\ge \gamma} x_\gamma=x=\inf_{\gamma\in\Gamma}\sup_{\gamma'\ge \gamma} x_\gamma=\limsup_\gamma x_\gamma.
\end{equation}
It is easy to see that if $(x_\gamma)_{\gamma\in\Gamma}$ satisfies (\ref{e3}) then it O$_1$-converges to $x$.  Moreover, if there exist subsets $\mm$ and $\nn$ of $\pp$ satisfying $\sup\mm=\inf\nn=x$, and such that for every $(m,n)\in\mm\times \nn$
\[x_\gamma\in[m,n]\quad \text{eventually},\]
then $\liminf_\gamma x_\gamma\ge \sup\mm=x$ and $\limsup_\gamma x_\gamma\le \inf\nn=x$.  This shows that when $\pp$ is a conditionally complete lattice all three notions of order convergence coalesce to the condition expressed in (\ref{e3}).  In this case we simply say $(x_\gamma)_{\gamma\in\Gamma}$ O-converges to $x$.

The following example shows that -- even when the poset is a lattice -- O$_2$-convergence does not imply O$_1$-convergence.  This question was first answered in \cite{MathewsAnderson1967} but the example that we shall give here is credited to D. Fremlin \cite[Ex. 2, pg. 140]{ShaeferBLattices}.

\begin{exmp}[O$_2$-convergence is not the same as O$_1$-convergence]
Let $X$ be an uncountable discrete space and let $\dd$ denote the family of all the cofinite subsets of $X$.  Let $X_0$ denote the one-point compactification of $X$.  This can be identified with the set $X\cup\{\infty\}$, equipped with the topology $\mathscr T:=2^X\cup \hat{\dd}$, where $\hat{\dd}:=\set{D\cup\{\infty\}}{D\in\dd}$.  The net $\bigl(\chi_{D}\bigr)_{D\in\hat{D}}$ belongs to the Banach lattice $C(X_0,\R)$ and  $\chi_D\downarrow 0$. So, if $(x_n)_{n\in\N}$ is an arbitrary sequence in $X_0$ such that all terms are distinct, the corresponding sequence of characteristic functions  $(f_n)_{n\in\N}$ is O$_2$-convergent to $0$.
On the other-hand, if $f\in C(X_0,\R)$ and  $f(x)\ge 1$ holds on some infinite subset of $X_0$ then, for every $\varepsilon>0$, there exists  $D\in \hat \dd$ such that $f(x)>1-\varepsilon$ holds on $D$.  So, if $(g_n)_{n\in\N}$ is a sequence in $C(X_0,\R)$ satisfying $f_n\le g_n\le g_{n+1}$ for every $n\in\N$ then,  for every $\varepsilon>0$, there exists an infinite subset $Y\subset X_0$ such that every $g_n$ is at least $1-\varepsilon$ on $Y$.  This implies the zero function cannot be the infimum of $(g_n)_{n\in\N}$ and therefore $(f_n)_{n\in\N}$ cannot be O$_1$-convergent to $0$.
\end{exmp}

This following example is due to E.S.Wolk \cite{Wolk1961} and shows that in general O$_2$-convergence is not the same as O$_3$-convergence.

\begin{exmp}[O$_3$-convergence is not the same as O$_2$-convergence]
Let $A:=\{a_n:n\in\N\}$ and $B:=\{b_n:n\in\N\}$ be two countable sets and let $\pp:=A\cup B\cup\{\mathbf 0,\mathbf 1\}$ equipped with the partial order defined by:
\[\mathbf 0\le x\le \mathbf 1\ (\forall x\in\pp)\qquad\text{and}\qquad x\le y\ \Leftrightarrow\ x=a_n,\, y=b_m\ (\exists n\le m).\]
It is clear that $\pp$ contains no infinite directed (or filtered) set.  On the other-hand, $\sup A=\mathbf 1$.  Thus, the sequence $(b_n)_{n\in\N}$ is O$_3$-convergent to $\mathbf 1$ but not O$_2$-convergent.
\end{exmp}

 For any subset $\dd$ of the poset $\pp$, define
 \begin{align*}
 \dd^+:=&\set{x\in\pp}{(\forall d\in\dd)\,x\ge d},\, \text{and}\\
 \dd^-:=&\set{x\in\pp}{(\forall d\in\dd)\,x\le d}.
 \end{align*}
   We recall that with respect to set-theoretic inclusion the set $\hat\pp:=\set{\dd\subset\pp}{(\dd^+)^-=\dd}$ becomes a complete lattice  satisfying:
 \begin{enumerate}
 \item $\aaa^-$  belongs to $\hat\pp$ for every $\aaa\subset\pp$.
 \item If $\set{\aaa_\iota}{\iota\in\mathds I}$ is a subset of $\hat\pp$ then $\sup_{\iota\in\mathds I}\aaa_\iota=\left(\bigcup_{\iota\in\mathds I}\aaa_\iota\right)^{+-}$ and $\inf_{\iota\in\mathds I}\aaa_{\iota}=\bigcap_{\iota\in\mathds I}\aaa_{\iota}$.
   \item $(\leftarrow,x]\in\hat\pp$ for every $x\in\pp$ and the function $\varphi:x\mapsto (\leftarrow,x]:\pp\to\hat\pp$ is an order isomorphism.
     \item $\varphi[\pp]$ is \emph{join-dense} and \emph{meet-dense} in $\hat\pp$, i.e.
   \[\sup\,\set{\varphi(x)}{x\in\pp,\,\varphi(x)\le a}\,=\, a\, =\, \inf\,\set{\varphi(x)}{x\in\pp,\,\varphi(x)\ge a},\]
   for every $a\in\hat\pp$.  In particular, it follows that $\varphi$ preserves all suprema and infima that exist in $\pp$; i.e. if $\aaa\subset\pp$ and $x\in\pp$, then $\sup\varphi[\aaa]=\varphi(x)$ if and only if $\sup \aaa=x$ (and dually, $\inf\varphi[\aaa]=\varphi(x)$ if and only if $\inf\aaa=x$).
   \item $\aaa^-=\inf\varphi[\aaa]$ and $\aaa^{+-}=\sup\varphi[\aaa]$ for every $\aaa\subset\pp$.
   \item $\inf\varphi[\aaa^+]=\sup\varphi[\aaa]$ and $\sup\varphi[\aaa^-]=\inf\varphi[\aaa]$ for every $\aaa\subset \pp$.
 \end{enumerate}
 The complete lattice $\hat\pp$ is called the \emph{Dedekind-MacNeille completion} of $\pp$.  It is characterized -- up to order-isomorphism -- as the unique complete lattice containing $\pp$ as a  simultaneously join-dense and meet-dense subset.

In  \cite{Birkhoff1967ThirdEd} one finds yet another definition for `order convergence' based on the embedding of the poset in its Dedekind-MacNeille completion:

\begin{defn}
A net $(x_\gamma)_{\gamma\in\Gamma}$ in a poset $\pp$ is said to \emph{O\textsuperscript{{\tiny{DM}}}-converge} to $x\in\pp$ if the net $\bigl(\varphi(x_\gamma)\bigr)_{\gamma\in\Gamma}$ O-converges to $\varphi(x)$ in $\hat \pp$.
\end{defn}

\begin{remark}
  Note that
\[x_\gamma\converges{O$^{\sss{\text{DM}}}$}x\,\Longleftrightarrow\liminf_\gamma\varphi(x_\gamma)=
\limsup_\gamma\varphi(x_\gamma)=\varphi(x).\]
The comment after the definition of order convergence found in \cite[pg. 244]{Birkhoff1967ThirdEd} shows that Birkhoff does not make a proper distinction between these modes of convergence, particularly he erroneously writes that O$_1$-convergence is equivalent to O\textsuperscript{\tiny{DM}}-convergence (compare to Theorem \ref{t1}).
\end{remark}

\begin{prop}\label{p1} Let $(x_\gamma)_{\gamma\in\Gamma}$ be a net in a poset $\pp$.  For any $x\in\ \pp$ the following three statements are equivalent:
\begin{enumerate}[{\rm(i)}]
\item $x_\gamma\converges{O$_3$}x$,
\item $\sup \bigcup_{\gamma\in\Gamma}E_\Gamma(\gamma)^-=x=\inf\bigcup_{\gamma\in\Gamma}E_\Gamma(\gamma)^+$,
\item $\bigcap_{\gamma\in\Gamma}E_\Gamma(\gamma)^{+-}=(\leftarrow,x]$ and $\bigcap_{\gamma\in\Gamma}E_\Gamma(\gamma)^{-+}=[x,\rightarrow)$.
    \end{enumerate}
\end{prop}
\begin{proof}
{\rm(i)} implies {\rm(ii)}. Assume that $x_\gamma\converges{O$_3$}x$ and let $\mm$, $\nn$ be subsets of $\pp$ satisfying $\sup \mm=\inf\nn=x$, and such that for every $(m,n)\in\mm\times\nn$ there exists $\gamma(m,n)$ such that $E_\Gamma(\gamma(m,n))\subset [m,n]$.  Then
\begin{align*}
\sup\varphi\left[\bigcup_{\gamma\in\Gamma}E_\Gamma(\gamma)^-\right]\,&=\,\sup\bigcup_{\gamma\in\Gamma}\varphi[E_\Gamma(\gamma)^-]\\
&\ge\,\sup\bigcup_{(m,n)\in\mm\times\nn}\varphi[E_\Gamma(\gamma(m,n))^-]\\
&=\,\sup_{(m,n)\in\mm\times\nn}\sup\varphi[E_\Gamma(\gamma(m,n))^-]\\
&=\,\sup_{(m,n)\in\mm\times\nn}\inf\varphi[E_{\Gamma}(\gamma(m,n))]\,\ge\,\sup_{(m,n)\in\mm\times\nn}\varphi(m)\,=\,\varphi(x),
\end{align*}
and therefore $\sup\bigcup_{\gamma\in\Gamma} E_\Gamma(\gamma)^-\,\ge\, x$.  Dually, one can show that $\inf\bigcup_{\gamma\in\Gamma}E_\Gamma(\gamma)^+\,\le\,x$.  On the other-hand, observe that if $s\in E_{\Gamma}(\gamma_s)^-$ and $t\in E_\Gamma(\gamma_t)^+$, then $s\le x_\gamma\le t$ holds for every $\gamma$ satisfying $\gamma\ge \gamma_s$ and $\gamma\ge\gamma_t$.  This shows that $\sup\bigcup_{\gamma\in\Gamma} E_\Gamma(\gamma)^-\,\le\,\inf\bigcup_{\gamma\in\Gamma}E_\Gamma(\gamma)^+$, and therefore
\[x\,\le\,\sup\bigcup_{\gamma\in\Gamma} E_\Gamma(\gamma)^-\,\le\,\inf\bigcup_{\gamma\in\Gamma}E_\Gamma(\gamma)^+\,\le\,x.\]
This proves that {\rm(i)} implies {\rm(ii)}.  To show that {\rm(ii)} implies {\rm(iii)} we first observe that:
\begin{equation}\label{e1}
t\in\bigcap_{\gamma\in\Gamma}E_\Gamma(\gamma)^{-+}\,\Leftrightarrow\,\varphi(t)\ge\sup\bigcup_{\gamma\in\Gamma}\varphi[E_\Gamma(\gamma)^-],
\end{equation}
and
\begin{equation}\label{e2}
t\in\bigcap_{\gamma\in\Gamma}E_\Gamma(\gamma)^{+-} \,\Leftrightarrow\,\varphi(t)\le\inf\bigcup_{\gamma\in\Gamma}\varphi[E_\Gamma(\gamma)^+].
\end{equation}
If $\sup\bigcup_{\gamma\in\Gamma}E_\Gamma(\gamma)^-=x$,  by (\ref{e1}) it follows that $t\in\bigcap_{\gamma\in\Gamma}E_\Gamma(\gamma)^{-+}$ if and only if $\varphi(t)\ge \varphi(x)$, i.e. $\bigcap_{\gamma\in\Gamma}E_\Gamma(\gamma)^{-+}=[x,\rightarrow)$.  Dually, if $\inf\bigcup_{\gamma\in\Gamma}E_\Gamma(\gamma)^+=x$, then $\bigcap_{\gamma\in\Gamma}E_\Gamma(\gamma)^{+-}=(\leftarrow,x]$.\\
{\rm(iii)} implies {\rm(i)}.  Let $\mm:=\bigcup_{\gamma\in\Gamma} E_\Gamma(\gamma)^-$ and $\nn:=\bigcup_{\gamma\in\Gamma}E_\Gamma(\gamma)^+$.  If $\bigcap_{\gamma\in\Gamma}E_\Gamma(\gamma)^{-+}=[x,\rightarrow)$, then
\[t\ge s\ (\forall s\in \mm)\Leftrightarrow  t\in\bigcap_{\gamma\in\Gamma}E_\Gamma(\gamma)^{-+}=[x,\rightarrow),\]
and if $\bigcap_{\gamma\in\Gamma} E_\Gamma(\gamma)^{+-}=(\leftarrow,x]$, then
\[t\le s\ (\forall s\in\nn)\Leftrightarrow t\in\bigcap_{\gamma\in\Gamma} E_\Gamma(\gamma)^{+-}=(\leftarrow,x].\]
This shows that $\sup \mm=\inf\nn=x$.  By definition of $\mm$ and $\nn$ it is clear that for every $(m,n)\in\mm\times\nn$ there exists $\gamma(m,n)\in\Gamma$ such that $m\in E_\Gamma(\gamma(m,n))^-$ and $n\in E_\Gamma(\gamma(m,n))^+$.  So $(x_\gamma)_{\gamma\in\Gamma}$ is O$_3$-convergent to $x$.
\end{proof}

\begin{remark}
In particular, Proposition \ref{p1} shows that the O$_3$-limit (and therefore the O$_2$-limit and the O$_1$-limit) of a net is uniquely determined when it exists.
\end{remark}

\begin{thm}\label{t1}
\begin{enumerate}[{\rm(i)}]
  \item In a poset, O$_3$-convergence is equivalent to O\textsuperscript{\tiny{DM}}-convergence.
  \item In a lattice,  O$_2$-convergence, O$_3$-convergence and O\textsuperscript{\tiny{DM}}-convergence are equivalent.
\end{enumerate}
\end{thm}
\begin{proof}
Let $(x_\gamma)_{\gamma\in\Gamma}$ be a net and $x$ and point in $\pp$.\\
{\rm(i)}~Observing that
\[\bigcap_{\gamma\in\Gamma}E_\Gamma(\gamma)^{+-}=\inf_{\gamma\in\Gamma}\sup\varphi[E_\Gamma(\gamma)]=\limsup_{\gamma}\varphi(x_\gamma)\]
and
\[\left(\bigcap_{\gamma\in\Gamma} E_\Gamma(\gamma)^{-+}\right)^-=\left(\bigcup_{\gamma\in\Gamma}E_\Gamma(\gamma)^-\right)^{+-}=\sup_{\gamma\in\Gamma}\inf\varphi[E_\Gamma(\gamma)]=\liminf_\gamma\varphi(x_\gamma),\]
one readily sees that the assertion follows by Proposition \ref{p1}.\\
{\rm(ii)}~Suffices to show that in a lattice O$_3$-convergence implies O$_2$-convergence.  In view of Proposition \ref{p1}, it is enough to show that $\bigcup_{\gamma\in\Gamma}E_\Gamma(\gamma)^-$ is directed and $\bigcup_{\gamma\in\Gamma}E_\Gamma(\gamma)^+$ is filtered.  We show that $\bigcup_{\gamma\in\Gamma}E_\Gamma(\gamma)^-$ is directed; the other assertion follows dually.  For $i=1,\dots,k$ suppose that $s_i\in E_\Gamma(\gamma_i)^-$. Then $\bigvee_{i=1}^k s_i\in E_\Gamma(\gamma)^-$ for every $\gamma$ satisfying $\gamma\ge \gamma_i$ for every $i=1,\dots,k$, i.e. $\bigvee_{i=1}^k s_i\in \bigcup_{\gamma\in\Gamma}E_\Gamma(\gamma)^-$.
\end{proof}

We recall that $(\pp,\le)$ is called \emph{monotone order separable} if every directed subset of $\pp$ with a supremum contains a sequence with the same supremum and, dually, every filtered subset  of $\pp$ having an infimum contains a  sequence with the same infimum.  In \cite{MathewsAnderson1967} it is proven that if a poset is monotone order separable, then a net that O$_2$-converges to a point is also O$_1$-convergent to the same point.

\begin{thm} In a monotone order separable poset, O$_2$-convergence implies O$_1$-convergence.
\end{thm}
\begin{proof}
Let $\mm$ and $\nn$ be subsets of $\pp$ satisfying $\sup \mm=\inf\nn=x$ and such that for every $(m,n)\in\mm\times\nn$ there exists $\gamma(m,n)$ satisfying $m\le x_\gamma\le n$ for every $\gamma\ge\gamma(m,n)$.  By the assumption, there is a monotonic increasing sequence $(m_i)_{i\in\N}$ in $\mm$, and a monotonic decreasing sequence $(n_i)_{i\in\N}$ in $\nn$, such that $\sup_{i\in\N}m_i=\inf_{i\in\N}n_i=x$.
Let $\gamma_1:=\gamma(m_1,n_1)$.  For $i\ge 2$ define inductively $\gamma_i$  by choosing an arbitrary $\gamma_i\in\Gamma$ satisfying $\gamma_i\ge \gamma_{i-1}$ and $\gamma_i\ge \gamma(m_i,n_i)$.  Let $\Gamma_0:=\set{\gamma}{\gamma\ngeq\gamma_1}$ and for $i\ge 1$ let $\Gamma_i:=\set{\gamma}{\gamma\ge\gamma_i,\ \gamma\ngeq\gamma_{i+1}}$.  If $\Gamma\mysetminus\bigcup_{i=0}^\infty\Gamma_i\neq\emptyset$, then $(x_\gamma)_{\gamma\in\Gamma}$ must be eventually constant equal to $x$, and therefore the assertion follows.  Suppose that $\set{\Gamma_i}{i\ge 0}$ is a partition of $\Gamma$ and $\set{\gamma}{\gamma\ge\gamma_i}=\bigcup_{k\ge i}\Gamma_k$.  Put $m_0$ and $n_0$ equal to any element of $\pp$ and define the nets $(m_\gamma)_{\gamma\in\Gamma}$ and $(n_\gamma)_{\gamma\in\Gamma}$ by setting $(m_\gamma,n_\gamma):=(m_i,n_i)$ if $\gamma\in\Gamma_i$.  Then, $(m_\gamma)_{\gamma\ge\gamma_1}$ is monotonic increasing to $x$, $(n_\gamma)_{\gamma\ge\gamma_1}$ is monotonic decreasing to $x$ and $m_\gamma\le x_\gamma\le n_\gamma$ for every $\gamma\ge\gamma_1$, i.e., eventually,  $m_\gamma\le x_\gamma\le n_\gamma$, $m_\gamma\uparrow x$ and $n_\gamma\downarrow x$.
\end{proof}

Let $i\in\{1,2,3\}$.  A subset $\xx$ of $\pp$ is said to be \emph{O$_i$-closed} if  there is no net in $\xx$ O$_i$-converging to a point outside of $\xx$.  The collection of O$_i$-closed sets comprises the closed sets for a topology, which we shall denote by $\tau_{\text{O$_i$}}(\pp)$.  Since O$_1$-convergence is `stronger' than O$_2$-convergence, and the later is yet `stronger' than O$_3$-convergence, it follows that \[\tau_{\text{O$_3$}}(\pp)\,\subset\,\tau_{\text{O$_2$}}(\pp)\,\subset\,\tau_{\text{O$_1$}}(\pp).\]
When $\pp$ is a complete lattice, all three topologies are equal; in this case we write $\tau_\text{O}(\pp)$.

\begin{thm}\label{t2}
Let $(\pp,\le)$ be a poset and let $\hat{\pp}$ denote its Dedekind-MacNeille completion.
\begin{enumerate}[{\rm(i)}]
\item $\tau_{\text{O$_1$}}(\pp)=\tau_{\text{O$_2$}}(\pp)$.
\item $\tau_{\text{O$_3$}}(\pp)\supset\tau_{\text{O}}(\hat{\pp})|_\pp$.
\item If $\pp$ is a lattice, $\tau_{\text{O$_1$}}(\pp)=\tau_{\text{O$_2$}}(\pp)=\tau_{\text{O$_3$}}(\pp)$.
\end{enumerate}
\end{thm}
\begin{proof}
{\rm(i)}~We need to show that every O$_1$-closed subset $\xx$ of $\pp$ is O$_2$-closed.  To this end, we suppose that $\xx\subset\pp$ and $(x_\gamma)_{\gamma\in\Gamma}$ is a net in $\xx$ that O$_2$-converges to $x\in\pp$.  Then there are subsets $\mm$ and $\nn$ of $\pp$, where $\mm$ is directed, $\nn$ is filtered, and satisfying that for every $(m,n)\in\mm\times\nn$ there exists $\gamma(m,n)\in\Gamma$ such that $x_\gamma\in[m,n]$ for every $\gamma\ge \gamma(m,n)$.  The product $\mm\times\nn$ becomes a directed set when equipped with the partial order $\preceq$  defined by
\[(m,n)\preceq (m',n')\ \Leftrightarrow\  m\le m'\,\text{ and }\, n'\le n.\]
The net $(m,n)\mapsto x_{\gamma(m,n)}:\mm\times\nn\to  \xx$ is O$_1$-convergent to $x$ because the net $(m,n)\mapsto m: \mm\times\nn\to\pp$ is monotonic increasing to $x$; the net $(m,n)\mapsto n:\mm\times\nn\to\pp$ is monotonic decreasing to $x$; and $x_{\gamma(m,n)}\in[m,n]$ for every $(m,n)\in\mm\times\nn$.     So, if $\xx$ is O$_1$-closed, then $\xx$ is O$_2$-closed.  This completes the proof of {\rm(i)}.

{\rm(ii)}~If $\xx\subset\pp$ is closed in $\pp$ w.r.t.  $\tau_{\text{O}}(\hat{\pp})$ and $(x_\gamma)_{\gamma\in\Gamma}$ is a net $\xx$ that O$_3$-converges to $x\in\pp$ then it is also O\textsuperscript{\tiny{DM}}-convergent to $x$ (by {\rm(i)} of Theorem \ref{t1}) and so $x\in \xx$.

{\rm(iii)}~This follows by {\rm(i)} and by {\rm(ii)} of Theorem \ref{t1}.
\end{proof}

In the following we make use of a construction by V.~ Olej\v cek \cite{Olejcek2000} to give an example that shows that the inclusion in {\rm(ii)} of Theorem \ref{t2} above may indeed be  proper.

\begin{exmp}
 Let $\Z':=\Z\mysetminus \{0\}$.  Endow $L:=\bigcup_{i\in\Z'}\{a(i),b(i)\}\cup\{0\}$ with the partial order $\le$ induced by:
\begin{itemize}
\item $a(i)\ge a(j)\ge 0$ and $b(i)\ge b(j)\ge 0$ for every $i,j\in\Z'$ satisfying $1\le i\le j$,
\item $a(i)\le a(j)\le 0$ and $b(i)\le b(j)\le 0$ for every $i,j\in\Z'$ satisfying $-1\ge i\ge j$,
\item $a(i)\le b(i)$ if $i\ge 1$ and $a(i)\ge b(i)$ if $i\le -1$.
\end{itemize}
It is easy to verify that $(L,\le)$ is a lattice satisfying
\[\limsup_{i\in\N}a(-i)=\limsup_{i\in\N}b(-i)=\liminf_{i\in\N}a(i)=\liminf_{i\in\N}b(i)=0.\]
  For every $k\in\N$ let $L_k$ be a copy of $L$.  Endow $\mathscr L':=\bigcup_{k\in\N} L_k\cup\{e\}$  with the  partial order $\preceq$ induced by:
\[ x\preceq y \ \Longleftarrow\  (\exists k\in\N)\ \begin{cases}
                             & (x,y\in L_k,\, x\le y) \mbox{ or} \\
                             & (x=a_k(-1),\  y=a_{k+1}(-1))\ \mbox{or}\\
                             & (x=a_{k+1}(1),\ y=a_{k}(1))\ \mbox{or}\\
                             & (x=a_k(-1),\ y=e)\ \mbox{or}\\
                             & (x=e,\ y=a_k(1)).
                            \end{cases}\]
                            One can verify that $\hat{\mathscr L}:=\mathscr L'\cup\{\mathbf{0},\mathbf{1}\}$ is a complete lattice.  Figure 1 shows the Hasse diagram of $\hat{\mathscr L}$.\footnote{For typographic reasons we write $ak{i}$ (resp. $bki$) instead of $a_k(i)$ (resp. $b_k(i)$) and $\bar{i}$ instead of $-i$.}
                            Define  $\bb:=\{b_k(i):k\in\N,\,i\in\Z'\}$ and $\oo:=\{0_k:k\in\N\}$. Then, $\mathscr L:=\hat{\mathscr L}\mysetminus \oo$ is a lattice and its Dedekind-MacNeille completion can be identified with $\hat{\mathscr L}$.  Being a lattice, one has $\tau_{\text{O$_1$}}(\mathscr L)=\tau_{\text{O$_2$}}(\mathscr L)=\tau_{\text{O$_3$}}(\mathscr L)$, so to prove our claim it is enough to show that $\tau_{\text{O$_1$}}(\mathscr L)\neq \tau_{\text{O}}(\hat{\mathscr L})|_{\mathscr L}$.  To this end, one needs to observe that in $\mathscr L$ every (nontrivial) monotonic net has to either be eventually in $L_k$  (and therefore converges to $0_k$) for some $k$ or be eventually in one of the sets $\{a_k(1):k\in\N\}$ or $\{a_k(-1):k\in\N\}$.  Hence, the set $\bb$ is closed w.r.t. $\tau_{\text{O$_1$}}(\mathscr L)$.  On the other-hand, every subset $\mathscr F\subset\hat{\mathscr L}$ that is closed w.r.t. $\tau_{\text{O}}(\hat{\mathscr L})$ and that contains $\bb$ must necessarily contain $\oo$ --  and therefore $e$ --  since $(0_k)_{k\in\N}$ O-converges to $e$.  This means that $\bb$ is not closed w.r.t. $\tau_{\text{O}}(\hat{\mathscr L})|_{\mathscr L}$.
\begin{figure}
\begin{center}
\psset{unit=0.95cm}
\begin{pspicture}(-8,-6)(8,6)

\dotnode(-8,6){I}
\uput[d](-8,6){${\scriptstyle\mathbf 1}$}

\dotnode(-8,-5.8){O}
\uput[d](O){${\scriptstyle\mathbf 0}$}

\dotnode(8,0){e}
\uput[u](8,0){${\scriptstyle e}$}

\dotnode(-8,4){{a11}}
\uput[u](-8,4){${\scriptstyle a11}$}
\dotnode(-7.9,3){{a12}}
\uput[l]({a12}){$\scriptstyle a12 $}
\dotnode(-7.75,2){{a13}}
\uput[l]({a13}){$\scriptstyle a13$}
\dotnode(-7.6,1.3){{a14}}
\uput[l](-7.6,1.2){$\scriptstyle a14$}
\dotnode(-6,-3.8){{a11*}}
\uput[l](-6,-3.7){$\scriptstyle a1\bar{1}$}
\dotnode(-6.1,-3){{a12*}}
\uput[l](-6.1,-2.85){$\scriptstyle a1\bar{2}$}
\dotnode(-6.2,-2.2){{a13*}}
\uput[l](-6.2,-2.1){$\scriptstyle a1\bar{3}$}
\dotnode(-6.4,-1.5){{a14*}}
\uput[l](-6.45,-1.4){$\scriptstyle a1\bar{4}$}
\dotnode(-6,4.6){{b11}}
\uput[r](-6,4.6){$\scriptstyle b11$}
\dotnode(-6,3.8){{b12}}
\uput[r](-6,3.8){$\scriptstyle b12$}
\dotnode(-6.1,2.7){{b13}}
\uput[r](-6.1,2.7){$\scriptstyle b13$}
\dotnode(-6.3,1.8){{b14}}
\uput[r](-6.3,1.8){$\scriptstyle b14$}
\dotnode(-8,-4.8){{b11*}}
\uput[l](-8,-4.8){$\scriptstyle b1\bar{1}$}
\dotnode(-8,-3.7){{b12*}}
\uput[l](-8,-3.7){$\scriptstyle b1\bar{2}$}
\dotnode(-7.9,-2.8){{b13*}}
\uput[l](-7.9,-2.8){$\scriptstyle b1\bar{3}$}
\dotnode(-7.7,-2){{b14*}}
\uput[l](-7.7,-2){$\scriptstyle b1\bar{4}$}
\dotnode(-7,0){{O1}}
\uput[r](-7,0){$\scriptstyle 0_1$}
\pscurve[showpoints=true]{-}({a11})({a12})({a13})({a14})
\pscurve[showpoints=true,linestyle=dashed]{-}({a14})({O1})({a14*})
\pscurve[showpoints=true]{-}({a14*})({a13*})({a12*})({a11*})
\pscurve[showpoints=true]{-}({b11})({b12})({b13})({b14})
\pscurve[showpoints=true, linestyle=dashed]{-}({b14})({O1})({b14*})
\pscurve[showpoints=true]{-}({b11*})({b12*})({b13*})({b14*})
\psline({a11})({b11})
\psline({a12})({b12})
\psline({a13})({b13})
\psline({a14})({b14})
\psline({a11*})({b11*})
\psline({a12*})({b12*})
\psline({a13*})({b13*})
\psline({a14*})({b14*})

\dotnode(-5,3.2){{a21}}
\uput[u](-5,3.2){$\scriptstyle a21$}
\dotnode(-4.9,2.4){{a22}}
\uput[l](-4.9,2.3){$\scriptstyle a22$}
\dotnode(-4.7,1.7){{a23}}
\uput[l](-4.7,1.68){$\scriptstyle a23$}
\dotnode(-4.4,1){{a24}}
\uput[l](-4.4,0.9){$\scriptstyle a24$}
\dotnode(-3,-3){{a21*}}
\uput[r](-3.05,-2.65){$\scriptstyle a2\bar{1}$}
\dotnode(-3.2,-2){{a22*}}
\uput[r](-3.2,-1.9){$\scriptstyle a2\bar{2}$}
\dotnode(-3.4,-1.4){{a23*}}
\uput[r](-3.4,-1.3){$\scriptstyle a2\bar{3}$}
\dotnode(-3.6,-0.8){{a24*}}
\uput[r](-3.6,-0.7){$\scriptstyle a2\bar{4}$}
\dotnode(-4,0){{O2}}
\uput[r](-4,0){$\scriptstyle 0_2$}
\dotnode(-2.8,3.8){{b21}}
\uput[r](-2.8,3.68){$\scriptstyle b21$}
\dotnode(-2.9,3.05){{b22}}
\uput[r](-2.9,3.05){$\scriptstyle b22$}
\dotnode(-3.1,2.2){{b23}}
\uput[r](-3.1,2.2){$\scriptstyle b23$}
\dotnode(-3.4,1.3){{b24}}
\uput[r](-3.4,1.3){$\scriptstyle b24$}
\dotnode(-5,-4){{b21*}}
\uput[l](-5,-3.9){$\scriptstyle b2\bar{1}$}
\dotnode(-4.9,-2.5){{b22*}}
\uput[l](-4.9,-2.5){$\scriptstyle b2\bar{2} $}
\dotnode(-4.8,-1.8){{b23*}}
\uput[l](-4.8,-1.8){$\scriptstyle b2\bar{3}$}
\dotnode(-4.5,-1){{b24*}}
\uput[l](-4.5,-1){$\scriptstyle b2\bar{4}$}
\pscurve[showpoints=true]{-}({a21})({a22})({a23})({a24})
\pscurve[showpoints=true,linestyle=dashed]{-}({a24})({O2})({a24*})
\pscurve[showpoints=true]{-}({a24*})({a23*})({a22*})({a21*})
\pscurve[showpoints=true]{-}({b21})({b22})({b23})({b24})
\pscurve[showpoints=true, linestyle=dashed]{-}({b24})({O2})({b24*})
\pscurve[showpoints=true]{-}({b21*})({b22*})({b23*})({b24*})
\psline({a21})({b21})
\psline({a22})({b22})
\psline({a23})({b23})
\psline({a24})({b24})
\psline({a21*})({b21*})
\psline({a22*})({b22*})
\psline({a23*})({b23*})
\psline({a24*})({b24*})

\dotnode(-2,2.6){{a32}}
\uput[u](-1.8,2.6){$\scriptstyle a31$}
\dotnode(-1.8,1.4){{a33}}
\uput[l](-1.8,1.4){$\scriptstyle a32$}
\dotnode(-1.6,0.7){{a34}}
\uput[l](-1.6,0.7){$\scriptstyle a33$}
\dotnode(0,-2.2){{a32*}}
\uput[r](0,-2.3){$\scriptstyle a3\bar{1}$}
\dotnode(-0.1,-1.5){{a33*}}
\uput[r](-0.1,-1.5){$\scriptstyle a3\bar{2}$}
\dotnode(-0.4,-0.8){{a34*}}
\uput[r](-0.4,-0.8){$\scriptstyle a3\bar{3}$}
\dotnode(-1,0){{O3}}
\uput[r](-1,0){$\scriptstyle 0_3$}

\dotnode(0,3){{b32}}
\uput[r](0,3){$\scriptstyle b31$}
\dotnode(-0.1,1.8){{b33}}
\uput[r](-0.1,1.7){$\scriptstyle b32$}
\dotnode(-0.3,1){{b34}}
\uput[r](-0.3,1){$\scriptstyle b33$}
\dotnode(-2.1,-3.2){{b32*}}
\uput[r](-2.1,-3.3){$\scriptstyle b3\bar{1}$}
\dotnode(-2,-2){{b33*}}
\uput[r](-2.05,-2.15){$\scriptstyle b3\bar{2}$}
\dotnode(-1.7,-1.1){{b34*}}
\uput[r](-1.8,-1.25){$\scriptstyle b3\bar{3}$}

\pscurve[showpoints=true]{-}({a32})({a33})({a34})
\pscurve[showpoints=true,linestyle=dashed]{-}({a34})({O3})({a34*})
\pscurve[showpoints=true]{-}({a34*})({a33*})({a32*})
\pscurve[showpoints=true]{-}({b32})({b33})({b34})
\pscurve[showpoints=true, linestyle=dashed]{-}({b34})({O3})({b34*})
\pscurve[showpoints=true]{-}({b32*})({b33*})({b34*})
\psline({a31})({b31})
\psline({a32})({b32})
\psline({a33})({b33})
\psline({a34})({b34})
\psline({a32*})({b32*})
\psline({a33*})({b33*})
\psline({a34*})({b34*})

\psline(O)({b11*})
\psline(O)({b21*})
\psline(O)({b32*})

\psline(I)({b11})
\psline(I)({b21})
\psline(I)({b32})

\pscurve({a11*})({a21*})({a32*})
\psline[linestyle=dashed]({a32*})(e)
\pscurve({a11})({a21})({a32})
\psline[linestyle=dashed]({a32})(e)

\end{pspicture}
\end{center}
{\tiny Figure 1}

\end{figure}

\end{exmp}

It is well-known that order convergence fails to be topological, in general; not even when $\pp$ is a complete lattice.   Characterizations of  posets having the property that  O$_2$-order convergence is topological  can be found in \cite{WangZhao2013,SunLiGuo2016,SunLi2018}.  For lattices, other characterizations have been thoroughly studied (in the language of filters) in \cite{Kent1964,Gingras1976,Erne1980,ErneRiecan1995}.  On the whole,  these characterizations relate to the requirement that the poset has `strong' distributivity properties.  For arbitrary posets, it is not possible to extract from a net that converges w.r.t.  $\tau_{\text{O}_i}(\pp)$ a subnet that is O$_i$-order convergent.  There is, however, one advantageous situation which has proved to be useful in \cite{BuChWe2012}.  Let us illustrate this by  proving a generalization.  For a regular cardinal $\lambda$, we shall call a function $f:\lambda\to \pp$ a \emph{ $\lambda$-sequence} in $\pp$. For any $i\in\{1,2,3\}$ we shall say that the subset $\xx$ of $\pp$ is O$^\lambda _i$-closed if there is no $\lambda$-sequence in $\xx$ that O$_i$-converges to a point outside of $\xx$.  The collection of all O$^\lambda_i$-closed sets comprises the closed sets for a topology, which we shall denote by $\tau^\lambda_{\text{O$_i$}}(\pp)$. The following inclusions are easy to verify.  (For two topologies $\tau$ and $\tau'$ on the same set let $\tau\rightarrow \tau'$ mean that $\tau$ is finer than $\tau'$.)

\[
\begin{array}{ccccc}
  \tau^\lambda_{\text{O$_3$}}(\pp) & \leftarrow & \tau^\lambda_{\text{O$_2$}}(\pp) & \leftarrow & \tau^\lambda_{\text{O$_1$}}(\pp)\\
  \downarrow &  & \downarrow &  &\downarrow \\
\tau_{\text{O$_3$}}(\pp) & \leftarrow & \tau_{\text{O$_2$}}(\pp) & = & \tau_{\text{O$_1$}}(\pp)
\end{array}
\]

\begin{lem}\cite[Claim 4.31.1, pg. 69]{Schimmerling2011}
Let $h:\lambda\to\lambda$ be a cofinal function.  Then there exists $L\subset\lambda$ such that the restriction $h|_L$ is cofinal and isotone.
\end{lem}

\begin{thm}  Let $\lambda$ be a regular cardinal and let  $f:\lambda\to \pp$ be a $\lambda$-sequence in a poset $\pp$ convergent  to $x\in\pp$ w.r.t. $\tau^\lambda_{\text{O$_i$}}(\pp)$ (where $i\in\{1,2,3\}$). Then, there exists an isotone and cofinal function $\hat h:\lambda\to\lambda$ such that the $\lambda$-sequence $f\circ\hat h$ is O$_{i}$-convergent to $x$.
\end{thm}
\begin{proof}
We can suppose that $f^{-1}\{x\}$ is bounded in $\lambda$; if not the assertion follows trivially.  We can therefore assume that $f(\alpha)\neq x$ for every $\alpha<\lambda$, i.e. the range $\xx$ of $f$ is not closed w.r.t. $\tau^\lambda_{\text{O$_i$}}(\pp)$.  By definition of $\tau^\lambda_{\text{O$_i$}}(\pp)$  there exists a $\lambda$-sequence $g:\lambda\to\xx$ that O$_i$-converges outside of $\xx$.   For every $t\in\xx$ the set $G_t:=g^{-1}\{t\}$ is bounded in $\lambda$; otherwise, $G_t$ is cofinal in $\lambda$ for some $t\in\xx$ and this  contradicts the assumption that $g$ is O$_i$-convergent to a point outside of $\xx$.   For every $\alpha<\lambda$ let $h(\alpha):=\min\set{\beta<\lambda}{f(\beta)=g(\alpha)}$. Note that $f(h(\alpha))=g(\alpha)$ for every $\alpha<\lambda$.

We show that $h[\lambda]$ is cofinal in $\lambda$.  Suppose, for contradiction, that there exists $\gamma<\lambda$ satisfying $h[\lambda]< \gamma$.  For every $\alpha<\gamma$ the set $H_\alpha:=h^{-1}\{\alpha\}$ is bounded in $\lambda$ because $H_\alpha\subset G_{f(\alpha)}$ and the latter is bounded in $\lambda$.  For every $\alpha<\gamma$ define $s(\alpha):=\min\set{\beta<\lambda}{H_\alpha<\beta}$.  In view of the regularity of $\lambda$ we have $\sup s[\gamma]< \lambda$ and therefore
\[\lambda=|\lambda|=\left|\bigcup_{\alpha<\gamma}H_\alpha\right|\le |\gamma|\,|\sup s[\gamma]|<\lambda,\]
which is a contradiction.  Therefore   $\sup h[\lambda]=\lambda$.

By the lemma there exists $L\subset \lambda$ such that $h|_L$ is cofinal and isotone.  Let $i:\lambda\to L$ be an order isomorphism and let $\hat{h}:=h\circ i$.  Then $f\circ \hat{h}$ is a $\lambda$-sequence in $\xx$ and
$f\circ\hat{h}(\alpha)=f(h(i(\alpha)))=g(i(\alpha))$, for every $\alpha<\lambda$.  Thus, $f\circ\hat{h}$ is O$_i$-convergent to $x$.
\end{proof}

In particular, when we set $\lambda=\omega$ we recover the following observation.  (This is well known for O$_1$ convergence; see \cite[Propoposition 2]{BuChWe2012}.)

\begin{cor}
If $(x_n)_{n\in\N}$ is a sequence in $\pp$ converging w.r.t. $\tau^\omega_{\text{O$_i$}}(\pp)$ to   some point $x$,  there exists a subsequence $(x_{n_i})_{i\in\N}$ that O$_i$-converges to $x$.
\end{cor}

\section{Applications}

In \cite{ChHaWe}, the order topology $\tau_{\text{O}_1}(\saM)$ associated with the self-adjoint part $\saM$ of a von Neumann algebra $M$ was studied and compared with the locally convex topologies that arise from the duality between $M$ and its predual.  Among other things, it was shown that on the self-adjoint part of the unit ball of a $\sigma$-finite von Neumann algebra, the restrictions of the order topology and the $\sigma$-strong topology are equal.    The arguments relied heavily on the assumption that the algebra is $\sigma$-finite.   In this section we shall apply Theorem \ref{t2} to obtain a partial improvement  to this.  First we shall illustrate the commutative case  to emphasize that even in this case, the non $\sigma$-finite case is not fully understood yet.

\subsection{The order topology on $L^\infty$.}   Let $(X,\Sigma,\mu)$ be a semi-finite measure space and let $L^\infty$ denote the Banach algebra of essentially bounded real-valued functions.
Other useful topologies on $L^\infty$, besides the topology $\tau_\infty$ induced by the norm $\norm{\infty}{ }$,  are:

\begin{enumerate}[{\rm(i)}]
  \item \emph{the topology of convergence in measure $\tau_\mu$}: Every $E\in\Sigma$ satisfying $\mu(E)<\infty$ defines an F-seminorm $\rho_E:f\mapsto \integral{|f|\wedge\chi_E}$ on $L^\infty$.  The Hausdorff and linear                                                                                                                                                                topology induced by the family $\set{\rho_E}{E\in\Sigma,\,\mu(E)<\infty}$ is the topology of convergence in measure (on sets of finite measure) and is denoted by $\tau_\mu$.
  \item \emph{the strong-operator topology $\sigma_p$ where $1\le p<\infty$}:  Every $g\in L^\infty$ induces a bounded linear operator on the Banach space $L^p$ of $p$-integrable real-valued functions defined by $f\mapsto fg$. In this way $L^\infty$ embeds isometrically (as a Banach algebra) in the space of bounded linear operators $B(L^p)$ and can therefore be endowed with the subspace topology induced by the strong-operator topology of $B(L^p)$ generated by the family of order continuous Riesz seminorms $f\mapsto \norm{p}{fg}$  as $g$ ranges in $L^p$.   Let $\sigma_p$ denote this locally convex-solid topology.
\item \emph{the weak topology $\sigma(L^\infty,L^1)$ and the Mackey topology $\tau(L^\infty,L^1)$, respectively}: The bilinear form $L^\infty\times L^1\to\R$ defined by $\langle f,g\rangle\mapsto\integral{fg}$ induces a duality and $L^\infty$ embeds isometrically (as a Banach space) in the dual space of $L^1$.  (We recall that this embedding is onto if and only if $(X,\Sigma,\mu)$ is localisable.)  Among the locally convex topologies that are consistent with this duality, the weak topology $\sigma(L^\infty,L^1)$ is the coarsest and the Mackey topology $\tau(L^\infty,L^1)$ is the finest.  $\sigma(L^\infty,L^1)$ is the topology of pointwise convergence on $L^1$ and $\tau(L^\infty,L^1)$ is the topology of uniform convergence on all absolutely convex $\sigma(L^\infty,L^1)$-compact subsets of $L^1$,  or equivalently, on all absolutely convex relatively $\sigma(L^\infty,L^1)$-compact subsets of $L^1$.
\end{enumerate}

With the partial order induced by the pointwise  order, $L^\infty$ is a Riesz space. Therefore the three modes of order convergence described in the previous section give rise to the same order topology $\tau_{\text{O}}(L^\infty)$.  It is easy to see that the topologies $\tau_\mu$,  $\sigma(L^\infty,L^1)$ and $\sigma_p$ are \emph{order-continuous topologies}  (i.e. coarser than  $\tau_{\text{O}}(L^\infty)$).

We shall need the following characterization of the subsets of $L^1$ that are relatively compact w.r.t. $\sigma(L^1,L^\infty)$.  Some care is required because under our assumptions $L^\infty$ may well be a proper subspace of the Banach dual $M:=(L^1)^\ast$ and in this case the weak topology $\sigma(L^1,M)$ is strictly finer than $\sigma(L^1,L^\infty)$.  Let {$\ca{\Sigma}$} be the space of all $\sigma$-additive $\R$-valued measures on $\Sigma$ with the variation-norm (equivalent to the sup-norm). To every $f\in \mathcal L^1$ we associate the $\R$-valued measure $\mu_f\in\ca{\Sigma}$ defined by {$\mu_f(E):=\rintegral{E}{f}$}.  The map $\Phi:f\mapsto \mu_f$ is an isometric embedding of $L^1$ into $\ca{\Sigma}$.
We denote by $\tau_p$ the product topology on $\ca{\Sigma}\subset \R^\Sigma$.

\begin{prop}\label{p2}
For $A\subset L^1$ the following conditions are equivalent:
\begin{enumerate}[{\rm(i)}]
\item  $A$ is relatively compact w.r.t. the weak topology $\sigma(L^1,M)$;
\item  $A$ is relatively compact w.r.t. $\sigma(L^1,L^\infty)$;
\item  $\Phi[A]$ is relatively compact in $\ca{\Sigma}$ w.r.t. $\tau_p$;
\item $\Phi[A]$ is uniformly exhaustive and pointwise bounded;
\item $A$ is uniformly integrable;
\item If $(y_\gamma)_{\gamma\in\Gamma}$ is a net in $L^\infty$ satisfying $y_\gamma\downarrow 0$, then
\[\lim_\gamma \sup_{u\in U}|\langle u,y_\gamma\rangle|=0.\]
\end{enumerate}
\end{prop}
 \begin{proof}
The implications {\rm(i)}$\Rightarrow${\rm(ii)}$\Rightarrow${\rm(iii)} are trivial.
{\rm(iii)}$\Rightarrow${\rm(iv)} follows from \cite[Theorem 2.6]{Ganssler1971}.
The implication {\rm(iv)}$\Rightarrow${\rm(v)} and the  equivalence of {\rm(i)}, {\rm(v)} and {\rm(vi)} can explicitly be found in \cite[246G, pg. 187 \& 247C, pg. 193 \& 246Y(k), pg191]{FremlinVol2}.
 \end{proof}

 By Proposition \ref{p2}  the notions of relative compactness in the weak topologies $\sigma(L^1,L^\infty)$ and $\sigma(L^1,M)$  are equivalent.  This has the following interesting consequence:

 \begin{cor}\label{c0}
  $\tau(L^\infty,L^1) =\tau(M,L^1)|_{L^\infty}$.
 \end{cor}

The equivalence  (ii)$\Leftrightarrow$(v) of Proposition \ref{p2} tells us that the family of all subsets of $L^1$ that are relatively compact  w.r.t. $\sigma(L^1,L^\infty)$  has  wide-ranging stability properties:
If $A$ is a uniformly integrable subset of $L^1$, then obviously its solid hull $s(A):=\{f\in L^1: |f|\leq |g| \text{ for some } g\in A\}$ is uniformly integrable. Therefore,  the convex hull $(s(A))$ is uniformly integrable. Since $(s(A))$ is absolutely convex, it follows with Proposition \ref{p2} that the absolutely convex solid hull of a relatively $\sigma(L^1,L^\infty)$-compact subset of $L^1$ is again  relatively $\sigma(L^1,L^\infty)$-compact.

 \begin{prop}\label{c1}
 \begin{enumerate}[{\rm(i)}]
  \item $\tau(L^\infty,L^1)$ is a locally convex-solid and order-continuous topology.
  \item $\tau(L^\infty,L^1)$ is the finest Hausdorff locally convex and order-continuous topology on $L^\infty$.
  \end{enumerate}
 \end{prop}
 \begin{proof}
{\rm(i)} It follows from the proceeding considerations that $\tau(L^\infty,L^1)$ is the topology of uniform convergence on all absolutely convex, solid, relatively $\sigma(L^1,L^\infty)$-compact subsets of $L^1$. This easily implies that the space $(L^\infty,\tau(L^\infty,L^1))$ has a $0$-neighbourhood of absolutely convex  solid sets, i.e.  $\tau(L^\infty,L^1)$ is a locally convex-solid topology.  From (ii)$\Rightarrow$(vi) of Proposition \ref{p2} it follows that $\tau(L^\infty,L^1)$ is order-continuous, i.e.  $\tau(L^\infty,L^1)$ is a Lebesgue topology in the terminology of \cite{AliBurk1978}.

 \rm{(ii)} Let $\tau$ be a Hausdorff locally convex and order-continuous topology on $L^\infty$.  First we show that every  linear functional $\varphi$ on $L^\infty$ that is continuous w.r.t. $\tau$ is also continuous w.r.t. $\sigma(L^\infty,L^1)$.
By our assumptions it follows that the set function $\nu:\Sigma\to\R$ defined by $\nu(E):=\varphi(\chi_E)$ is $\sigma$-additive.  We would like to apply the Radon-Nikod\'ym Theorem.  Clearly, $\nu$ is absolutely continuous w.r.t. $\mu$.  Let us verify that $\nu$ is $\sigma$-finite w.r.t. $\mu$.    The assumption of semi-finiteness implies that for every $E\in\Sigma$,  the directed set $\{\chi_F:F\in\Sigma,F\subset E,\mu(F)<+\infty\}$ increases to $\chi_E$.  Therefore the net $\{\nu(F):F\in\Sigma,F\subset E, \mu(F)<+\infty\}$ converges to $\nu(E)$, by the order-continuity of $\varphi$.   Every disjoint system $\dd$ in $\Sigma$ satisfying $\nu(D)\ne 0$ is countable; so if $\dd$ is a maximal disjoint system in $\Sigma$ such that $\mu(D)<+\infty$ and $\nu(D)\neq 0$ for every $D\in\dd$, then $\dd$ is countable.  Note that $\nu(E)=0$ for every $E\subset X\mysetminus \cup \dd$, i.e. $\nu$ is $\sigma$-finite w.r.t. $\mu$.
The Radon-Nikod\'ym Theorem can therefore be applied to yield a function $f\in L^1$ satisfying $\integral{\chi_E f}=\nu(E)$ for every $E\in\Sigma$.  But for every $g\in L^\infty_+$ (i.e. $g(x)\ge 0$ \muaew)  there is a sequence $(s_n)_{n\in\N}$ in $S:=\spn\{\chi_E:E\in\Sigma\}$ such that $s_n\uparrow g$.  This implies that  $\varphi(g)=\integral{gf}$ for every $g\in L^\infty_+$ and therefore $\varphi(g)=\integral{gf}$ for every $g\in L^\infty$.     In particular,  $\varphi$ is continuous w.r.t. $\sigma(L^\infty,L^1)$ and therefore the dual of $(L^\infty,\tau)$ can be identified with a  subspace $G$ of $L^1$.  If $K\subset G$ is compact w.r.t. $\sigma(G,L^\infty)$, it is also compact w.r.t. $\sigma(L^1,L^\infty)$; so $\tau\subset\tau(L^\infty, G)\subset \tau(L^\infty, L^1)$.

 \end{proof}

\begin{thm}\label{t4}  Let $(X,\Sigma,\mu)$ be a semi-finite measure space.
\begin{enumerate}[{\rm(i)}]
\item For every $1\le p<q<\infty$
\begin{align*}
&\sigma(L^\infty,L^1)\,\subset\, \sigma_p\,\subset\,\sigma_q\,\subset\,\tau(L^\infty,L^1)\,\subset\,\tau_{\infty},\ \text{and}\\
&\tau_\mu\,\subset\,\sigma_p,
\end{align*}
and  -- unless $L^\infty$ is finite-dimensional -- all of these inclusions are proper.
\item $\tau(L^\infty,L^1)\,\subset\,\tau_{\text{O}}(L^\infty)\,\subset\,\tau_{\infty}$ and $\tau_{\text{O}}(L^\infty)=\tau_{\infty}$ if and only if $L^\infty$ is finite-dimensional.
\item  If $(X,\Sigma,\mu)$ is $\sigma$-finite, the restrictions of $\tau_\mu$ and $\tau_{\text{O}}(L^\infty)$ to bounded parts of $L^\infty$ are equal.
\end{enumerate}
\end{thm}
\begin{proof}

{[$\sigma(L^\infty,L^1)\subset\sigma_1$]} For every $f\in  L^\infty$ and $g\in L^1$ one has $\left|\integral{fg}\right|\le\integral{|fg|}$.

{[$\sigma_p\subset \sigma_q$ where $1\le p\le q$]}  It is enough to show that for every $g\in  L^p$ and $\varepsilon>0$ there are $h\in L^q$ and $\delta>0$ such that $\norm{p}{fg}<\varepsilon$ holds for every $f\in L^\infty$ satisfying $\norm{q}{fh}<\delta$.  Fix an arbitrary $0\neq g\in L^p$.  (We can suppose that $g(x)\ge 0$ \muaew.)  First observe that for $f\in L^\infty$, $f(x)\ge 0$ \muaew,  one has: $fg^{\frac{p}{q}}\in L^q$, $f^{p-1}g^{\frac{p}{q^\ast}}\in L^{q^\ast}$ (where $q^\ast$ equals the conjugate exponent of $q$) and, by the H\"older inequality
\begin{equation}\label{e4}
\norm{p}{fg}^p \,=\, \norm{1}{f^p g^p} \,\le\, \norm{q}{fg}^{\frac{p}{q}}\,\norm{q^\ast}{f^{p-1} g^{\frac{p}{q^\ast}}}\, \le\, \norm{q}{fg^{\frac{p}{q}}}\,\norm{\infty}{f}^{p-1}\,\norm{p}{g}^{\frac{p}{q^\ast}}.
\end{equation}
Set $h:=g^{\frac{p}{q}}$ and
\[\delta:=(\varepsilon/2)^p\,\min\left( \norm{p}{g}^{\frac{-p}{q^\ast}}, \sqrt[q]{\varepsilon/2}\right).\]
Suppose that $f\in L^\infty$ satisfies $\norm{q}{fh}<\delta$.  Let $A:=\set{x\in X}{|f(x)|>1}$ and $B:=X\mysetminus A$.  Then
\[\norm{p}{\chi_A f g}^p=\rintegral{A}{|f|^pg^p}\le\integral{|f|^qh^q}<\delta^q\le(\varepsilon/2)^p\]
and,  by virtue of (\ref{e4}),
\[\norm{p}{\chi_B f g}^p\le \norm{q}{\chi_Bfg^{\frac{p}{q}}}\,\norm{p}{g}^{\frac{p}{q^\ast}}\le  \norm{q}{fh}\,\norm{p}{g}^{\frac{p}{q^\ast}}<(\varepsilon/2)^p.\]
This implies that $\norm{p}{fg}\le \norm{p}{\chi_Afg}+\norm{p}{\chi_B fg}<\varepsilon$.

{[$\sigma_q\subset\tau(L^\infty,L^1)$ and $\tau(L^\infty,L^1)\subset \tau_{\text{O}}(L^\infty)$]} These follow immediately by Proposition \ref{c1}.

{[$\tau_{\text{O}}(L^\infty)\subset \tau_\infty$]}  If $(x_n)_{n\in\N}$ is a sequence in $L^\infty$ satisfying $\lim_{n\to\infty}\Vert x_n\Vert_\infty= 0$, then
\[-\inf_{k\ge n}\norm{\infty}{x_k}\,\mathds 1\le x_n\le \sup_{k\ge n}\norm{\infty}{x_k}\,\mathds 1,\]
i.e. $(x_n)_{n\in\N}$ is O$_1$-convergent to $0$.

{[$\tau_\mu\subset\sigma_1$]}  If $F\in\Sigma$ and $\mu(F)<\infty$, then $\chi_F\in  L^1$ and  $\integral{|f|\wedge \chi_F}\le \integral{|f|\chi_F}$.

Now we suppose that $L^\infty$ is infinite-dimensional.  To this end we suppose that $(A_n)_{n\in\N}$ is a disjoint sequence in $\Sigma$ satisfying $0<\mu(A_n)<\infty$ for every $n$.  The map $\imath:c_0\to  L^\infty:(\lambda_n)_{n\in\N}\mapsto\sum_{n=1}^{\infty}\lambda_n \chi_{A_n}$ is a Banach space embedding of $c_0$ into $L^\infty$.  Denote by $E_0$ the range of $\imath$.

{[$\sigma(L^\infty,L^1)\neq \sigma_1$]}     Since the topology $\sigma_1$ is locally solid, it is enough to show that $\sigma(L^\infty,L^1)$ is not locally solid.  Via the embedding $\imath$, the topology $\sigma(c_0,\ell^1)$ can be identified with the restriction of $\sigma(L^\infty,L^1)$ to $E_0$.  If $\sigma(L^\infty,L^1)$ was locally solid, then $\sigma(c_0,\ell^1)$ is locally solid as well.  This contradicts \cite[Theorem 6.9, pg. 42]{AliBurk1978}.

{[$\sigma_p\neq \sigma_q$.  From this it follows also that $\sigma_q\neq \tau(L^\infty,L^1)$.]}  Fix $1<\alpha<q/p$ and define $e_n:=\sqrt[q]{n^\alpha}\chi_{A_n}$.   To prove that $\sigma_p\neq \sigma_q$ we first show that $0$ fails to be  a limit point of $F:=\{e_n:n\in\N\}\subset \mathcal L^\infty$ with respect to $\sigma_q$ and then we go showing that $0$ is a limit point of the set $F$ w.r.t.  $\sigma_p$. The function
\[g=\sum_{n=1}^\infty\frac{1}{\sqrt[q]{n^{\alpha}\mu(A_n)}}\chi_{A_n}\]
belongs to $L^q$ and  $\norm{q}{e_ng}=1$ for every $n\in\N$. So $U:=\set{f\in L^\infty}{\norm{q}{fg}<1}$ is a $\sigma_q$-open neighbourhood of $0$ satisfying $U\cap F=\emptyset$.
On the other-hand let $h\in L^p$ and $\varepsilon>0$ be given.  Let $\lambda_n=\rintegral{A_n}{|h|^p}$.  Then $\sum_{n=1}^{\infty}\lambda_n<\infty$, so there exists an increasing sequence $(n_k)_{k\in\N}$ such that $n_k\lambda_{n_k}\leq 1$  for every $k\in\N$. Since $\alpha p-q<0$, we can find sufficiently large $k\in \N$ satisfying $\sqrt[q]{n_k^{\alpha p-q}}<\varepsilon$.
Then
\[
\norm{p}{e_{n_k}h}^p\,=\,\rintegral{A_{n_k}}{\biggl(\sqrt[q]{n_k^\alpha}\
 |h|\biggr)^p}\,=\,\sqrt[q]{n_k^{\alpha p}}\,\lambda_{n_k}\,=\,\sqrt[q]{n_k^{\alpha p-q}}\,(n_k\lambda_{n_k})\,<\, \varepsilon,\]
i.e. $0$ is a limit point of $F$ w.r.t.  $\sigma_p$.


{[$\tau_{\text{O}}(L^\infty)\neq \tau_\infty$.  From this it follows also that $\tau(L^\infty,L^1)\neq \tau_\infty$.]}  Let $B_n:=\bigcup_{k\ge n}A_k$.  Observe that  $\chi_{B_n}\downarrow 0$ in $L^\infty$, i.e.  $\left(\chi_{B_n}\right)_{n\in\N}$ converges to $0$ w.r.t. $\tau_{\text{O}}(L^\infty)$ but not w.r.t. $\tau_\infty$.

{[$\tau_\mu\neq \sigma_1$]}  The sequence $\left(n^2\chi_{A_n}\right)_{n\in\N}$ satisfies $\rho_E(n^2\chi_{A_n})\to 0$, as $n\to\infty$, for every $E\in\Sigma$ satisfying $\mu(E)<+\infty$.  On the other-hand, if we let $1<\alpha<2$, the function $h:=\sum_{k=1}^{\infty}\frac{1}{\mu(A_k)k^\alpha}\chi_{A_k}$ belongs to $L^1$ and $\integral{n^2\chi_{A_n} h}=n^{2-\alpha}$, i.e. the sequence $\left(n^2\chi_{A_n}\right)_{n\in\N}$ converges to $0$ w.r.t. $\tau_\mu$ but not w.r.t. $\sigma_1$.

{\rm(iii)} We recall that if $\msp$ is $\sigma$-finite, $\tau_\mu$ is metrizable and for every sequence $(f_n)_{n\in\N}$ of measurable functions converging in measure to the function $f$,  there is a subsequence $(f_{n_k})_{k\in\N}$ that converges pointwise \muaew to $f$.  This implies that when $\msp$ is $\sigma$-finite,  every order-closed subset of the unit ball of $L^\infty$ is closed w.r.t. $\tau_\mu$.
\end{proof}

By {\rm(iii)} of Theorem \ref{t3} it follows, for a $\sigma$-finite measure, that $\tau_{\text{O}}(L^\infty)$ and $\tau(L^\infty,L^1)$ are equal if and only if $\mu$ is purely atomic.

 In Theorem \ref{t5} we give a general condition under which the topologies $\tau_{\text{O}}(L^\infty)$ and $\tau(L^\infty,L^1)$ are different.  In its proof we shall make use of the following lemma.  For $A\in\Sigma$ let $L^\infty(A):=\{f\chi_A:f\in L^\infty\}$ and $L^1(A):=\{f\chi_A:f\in L^1\}$.  It can easily been verified that $L^\infty(A)$ and $L^1(A)$ are isomorphic (as Banach lattices) to $L^\infty(A,\Sigma_A,\mu_A)$ and $L^1(A,\Sigma_A,\mu_A)$, respectively, where $\Sigma_A:=\{A\cap E: E\in\Sigma\}$ and $\mu_A:=\restr{\mu}{\Sigma_A}$.  In what follows we shall not distinguish between $L^\infty(A)$ and $L^\infty(A,\Sigma_A,\mu_A)$, and between $L^1(A)$ and $L^1(A,\Sigma_A,\mu_A)$.

\begin{lem}
Let $A\in\Sigma$.
\begin{enumerate}[{\rm(i)}]
\item The restriction of $\sigma(L^1,L^\infty)$ to $L^1(A)$ (and the restriction of $\tau(L^\infty,L^1)$ to $L^\infty(A)$) is equal to the topology $\sigma(L^1(A),L^\infty(A))$ (resp. $\tau(L^\infty(A),L^1(A))$) arising from the duality $\langle L^1(A),L^\infty(A)\rangle$.
\item The order topology $\tau_{\text{O}}(L^\infty(A))$ is equal to the restriction of $\tau_{\text{O}}(L^\infty)$ to $L^\infty(A)$.
    \end{enumerate}
\end{lem}
\begin{proof}
The first assertion requires merely a straightforward verification.  The second follows by observing that the map $L^\infty\to L^\infty(A):f\mapsto f\chi_A$ is isotone, idempotent, and preserves suprema and infima; so it is  closed and continuous w.r.t. $\tau_{\text{O}}(L^\infty)$ and $\tau_{\text{O}}(L^\infty(A))$.
\end{proof}

By the Lemma it follows that if $\tau_{\text{O}}(L^\infty)=\tau(L^\infty,L^1)$, then $\tau_{\text{O}}(L^\infty(A))=\tau(L^\infty(A),L^1(A))$ holds for every $A\in \Sigma$.

\begin{thm}\label{t5}
Let $A\in\Sigma$ satisfy $\mu(A)\ne 0$ and such that it contains no $\mu$-atoms.  Then
$\tau_{\text{O}}(L^\infty)$ and $\tau(L^\infty,L^1)$ are not equal.
\end{thm}
\begin{proof}

 In view of the semi-finiteness assumption, we can assume that $\mu(A)<+\infty$.  We shall exhibit a set $F\subset L^\infty(A)$ that is closed w.r.t. $\tau_{\text{O}}(L^\infty(A))$ but not closed w.r.t. $\tau(L^\infty(A),L^1(A))$.

 Let $A_{n\,i}$ ($n\in\N$, $i=1,\dots,2^n$) be a tree in $\Sigma$ such that $A$ is the disjoint union of $A_{1\,1}$ and $A_{1\,2}$, and $A_{n-1\,i}$ is the disjoint union of $A_{n\,2i-1}$ and $A_{n\,2i}$.  We can choose these sets such that $\mu(A_{n\,i})=2^{-n}\mu(A)$.  Set $A_n:=\bigcup_{i=1}^{2^{n-1}}A_{n\,2i}$ and $B_n:=A_{n\,1}$. We shall show that
 \[F:=\left\{\frac{1}{m}\chi_{A_n}+m\chi_{B_n}:m,n\in\N\right\}\]
 is closed w.r.t. $\tau_{\text{O}}(L^\infty(A))$, but $0$ is in the $\tau(L^\infty(A),L^1(A))$-closure of $F$.

 $F$ is $\tau_{\text{O}}(L^\infty(A))$-closed:  $L^\infty(A)$ is a conditionally complete lattice and  monotone order separable.  By \cite[Prop. 3]{BuChWe2020} it suffices to show that for every O-convergent sequence $(f_j)_{j\in\N}$ in $F$, its order limit $f$ belongs to $F$.  Let \[f_j=\frac{1}{m_j}\chi_{A_{n_j}}+m_j\chi_{B_{n_j}}.\]
 Since $(f_j)_{j\in\N}$ is O-convergent, it is norm bounded and therefore, by passing to a subsequence, we may suppose that there exists $m\in\N$ such that $m_j=m$ for every $j\in\N$.  Passing again to a subsequence we may assume that $(n_j)_{j\in\N}$ is constant or strictly increasing.  We can suppose that we have the second case.  The sequence $(m\chi_{B_{n_j}})_{j\in\N}$ decreases to $0$ in $L^\infty(A)$.  Therefore $\bigl(m^{-1}\chi_{A_{n_j}}\bigr)_{j\in\N}$ order converges to $f$.  This
 implies that $\bigl(m^{-1}\chi_{A_{n_j}}\bigr)_{j\in\N}$ is convergent w.r.t. the F-seminorm $\rho_A$, and therefore $\mu(A_{n_j}\,\triangle\,A_{n_{j+1}})\to 0$ as $j\to \infty$.  Since $\mu(A_{n_j}\,\triangle\,A_{n_{j+1}})=\mu(A)/2$ for every $j\in\N$ satisfying $n_{j+1}\neq n_j$, it follows that $(f_j)_{j\in\N}$ has a subsequence which is constant, i.e. $f\in F$.

 $0$ belongs to the $\tau(L^\infty(A),L^1(A))$-closure of $F$:  Let $K\subset L^1(A)$ be absolutely convex and relatively compact w.r.t. $\sigma(L^1(A),L^\infty(A))$. We show that there exists $f\in L^\infty(A)$ such that
 $\gamma_K(f):=\sup_{g\in K}|\langle f,g\rangle|\le 1$.  First, choose $m\in\N$ such that $m\ge 2\gamma_K(\chi_A)$.  Since $B_n\downarrow \emptyset$, and $K$ is uniformly integrable, there exists $n\in\N$ such that $\gamma_K(\chi_{B_n})\le 1/(2m)$.  Then $f:=\frac{1}{m}\chi_{A_n}+m\chi_{B_n}$ belongs to $F$ and
\[\gamma_K(f)\le \frac{1}{m}\gamma_K(\chi_{A_n})+m\gamma_K(\chi_{B_n})\le \frac{1}{m}\gamma_K(\chi_A)+m\gamma_K(\chi_{B_n})\le 1.\]

\end{proof}

\subsection{The order topology on $\saM$.}  Let $M$ denote a von Neumann algebra with predual $M_\ast$.  Denote by $\saM$ the self-adjoint part of $M$, and by $M_\ast^+$ the positive cone of $M_\ast$.
Between the restrictions of the weak$^\ast$-topology $\sigma(M,M_\ast)$  and the  Mackey topology
 $\tau(M,M_\ast)$, on $\saM$ lies the restriction of the $\sigma$-strong topology $s(M,M_\ast)$ determined by the family of seminorms $\set{\rho_\psi}{\psi\in M_{\ast}^+}$ where
$\rho_{\psi}(x)=\sqrt{\psi(x^\ast x)}$. $\saM$ is a real vector space and when endowed with the partial order $\leq$ induced by the cone $M^+:=\set{x^\ast x}{x\in M}$ it gets the structure of an ordered vector space.   In general $M_{sa}$ is far from being a Riesz space;  in fact, in \cite{Sherman1951} it is shown that if $\saM$ is a lattice then $M$ is abelian.  The following theorem can be found in \cite{ChHaWe}.

\begin{thm}\label{t3} Let $M$ be a von Neumann algebra.
\begin{enumerate}[{\rm(i)}]
\item $
\sigma(M,M_\ast)\,\subset\, s(M,M_\ast)\,\subset\,\tau(M,M_\ast)\,\subset\,\ttt_{\infty}$
and  -- unless $M$ is finite-dimensional -- all of these inclusions are proper.
\item $\tau(M,M_\ast)|_{\saM}\,\subset\,\tau_{\text{O}}(\saM)\,\subset\,\ttt_{\infty}|\saM$ and $\tau_{\text{O}}(\saM)=\ttt_{\infty}|\saM$ if and only if $M$ is finite-dimensional.
\item If $M$ is $\sigma$-finite, $\tau(M,M_\ast)|_{\saM}=\tau_{\text{O}}(\saM)$ if and only if the unit ball of $M$ is $s(M,M_\ast)$-compact if and only if every von Neumann subalgebra of $M$ is atomic (i.e. $M$ is purely atomic).
\item  If $M$ is $\sigma$-finite, the restrictions of $s(M,M_\ast)$ and $\tau_{\text{O}}(\saM)$ to bounded parts of $\saM$ are equal.
\end{enumerate}
\end{thm}

In the following theorem we show that in {\rm(iii)} of Theorem \ref{t3} -- even if we don't assume commutativity -- the condition of $\sigma$-finiteness is not a necessary one.  We show that when $M$ is atomic, then the restrictions of $s(M,M_\ast)$ and $\tau_{\text{O}}(\saM)$ to bounded parts of $\saM$ are equal.  The proof relies on our observation of Theorem \ref{t2}.

\begin{thm}
  If $M$ is an atomic von Neumann algebra, then the restrictions of $s(M,M_\ast)$ and $\tau_{\text{O}}(\saM)$ to bounded parts of $\saM$ are equal.
\end{thm}
\begin{proof}
Suppose that $(x_\gamma)_{\gamma\in\Gamma}$ is a net in the unit ball of $\saM$ that converges to $x$ w.r.t. $s(M,M_\ast)$.  We show that the net  $(x_\gamma)_{\gamma\in\Gamma}$  is O$_2$-convergent to $x$.  This would imply that $(x_\gamma)_{\gamma\in\Gamma}$ converges to $x$ w.r.t. $\tau_{\text{O$_2$}}(\saM)$.  But the latter is equal to $\tau_{\text{O$_1$}}(\saM)$ by Theorem \ref{t2}.

By assumption we can find an increasing net $(e_\lambda)_{\lambda\in\Lambda}$ of projections of finite rank in $M$ converging to $\mathds 1$ w.r.t. $s(M,M_\ast)$.  Let $d_\lambda$ denote the (linear) dimension of $e_\lambda M e_\lambda$.  Then $(d_\lambda^{-1})_{\lambda_\in\Lambda}$ is decreasing and convergent to $0$ and therefore the net $(p_\lambda)_{\lambda\in\Lambda}$ defined by
\[p_\lambda:=d_\lambda^{-1}\,\mathds 1\,+\,(\mathds 1-e_\lambda)\]
is decreasing to $0$ in $\saM$.  Fix an arbitrary $\lambda\in\Lambda$.  Since $e_\lambda$ is of finite rank, $e_\lambda M e_\lambda$ is a finite dimensional algebra over $\C$; hence $eMe$ admits only one topology compatible with the vector space structure.  Therefore, $(e_\lambda y_\gamma e_\lambda)_{\gamma\in\Gamma}$ is norm convergent to $0$ and, moreover, since $(y_\gamma)_{\gamma\in\Gamma}$ is bounded, we also have that $(e_\lambda y_\gamma^2 e_\lambda)_{\gamma\in \Gamma}$ is norm convergent to $0$. Thus, there exists $\gamma_0\in\Gamma$ such that
\[\Vert e_\lambda y_\gamma e_\lambda\Vert+2\Vert e_\lambda y_\gamma^2 e_\lambda\Vert\,<\,d_\lambda^{-1}\quad\text{for every }\gamma\ge \gamma_0.\]
Then
\[\Vert e_\lambda y_\gamma e_\lambda+e_\lambda y_\gamma (\mathds 1-e_\lambda)+(\mathds 1-e_\lambda)y_\gamma e_\lambda\Vert<d_\lambda^{-1}\quad\text{for every }\gamma\ge \gamma_0,\]
and therefore
\[-p_\lambda\,\le\,y_\gamma= e_\lambda y_\gamma e_\lambda+e_\lambda y_\gamma (\mathds 1-e_\lambda)+(\mathds 1-e_\lambda)y_\gamma e_\lambda+(\mathds 1-e_\lambda)y_\gamma(\mathds 1-e_\lambda)\,\le\,p_\lambda,\]
for every $\gamma\ge \gamma_0$.  This shows that $(y_\gamma)_{\gamma\in\Gamma}$ is O$_2$-convergent to $0$.   From this follows that $(x_\gamma)_{\gamma\in\Gamma}$ is O$_2$-convergent to $x$.
\end{proof}

\begin{prob}
Are the restrictions of $s(M,M_\ast)$ and $\tau_{\text{O}}(\saM)$ to bounded parts of $\saM$ equal for every von Neumann algebra $M$?  If not, is it possible to characterize  these von Neumann algebras?
\end{prob}

\bibliographystyle{amsplain}
\providecommand{\bysame}{\leavevmode\hbox to3em{\hrulefill}\thinspace}
\providecommand{\MR}{\relax\ifhmode\unskip\space\fi MR }
\providecommand{\MRhref}[2]{%
  \href{http://www.ams.org/mathscinet-getitem?mr=#1}{#2}
}
\providecommand{\href}[2]{#2}

\end{document}